\documentclass[reqno,11pt]{article}
\usepackage{amssymb,amsmath,latexsym}
\usepackage{amsfonts}
\usepackage{txfonts}
\usepackage{dsfont}

\newtheorem{thm}[equation]{Theorem}

\newtheorem{lem}[equation]{Lemma}

\newtheorem{cor}[equation]{Corollary}

\newtheorem{prop}[equation]{Proposition}

\newtheorem{defn}[equation]{Definition}

\newtheorem{rem}[equation]{Remark}

\def\sideremark#1{\ifvmode\leavevmode\fi\vadjust{\vbox to0pt{\vss
 \hbox to 0pt{\hskip\hsize\hskip1em
\vbox{\hsize2cm\tiny\raggedright\pretolerance10000 
 \noindent #1\hfill}\hss}\vbox to8pt{\vfil}\vss}}} 

\newenvironment{proof}[1][Proof]{\noindent\textbf{#1.} }{\ \rule{0.5em}{0.5em}}

\begin{document}
\title{Locally symmetric homogeneous Finsler spaces\footnote{The first author is supported by NSFC (no. 10671096, 10971104) and SRFDP of China}}
\author{Shaoqiang Deng\,$^1$ and Joseph A. Wolf\,$^2$\\
\\
$^1$ School of Mathematical Sciences and LPMC\\
Nankai University, Tianjin 300071\\ People's Republic of China\\
email: dengsq@nankai.edu.cn\\
\\
$^2$ Department of Mathematics\\ University of California, Berkeley\\
  Berkeley, CA 94720-3840, USA\\
email: jawolf@math.berkeley.edu}
\date{}

\maketitle

\begin{abstract}
Let $(M, F)$ be a connected Finsler space and $d$ the distance function of
$(M, F)$. A Clifford translation is an isometry $\rho$ of $(M, F)$ of
constant displacement, in other words such
that $d(x, \rho(x))$ is a constant function on $M$. In this paper we
consider a connected simply connected symmetric Finsler space and a
discrete subgroup $\Gamma$ of the full group of isometries.  We prove that
the quotient manifold $(M, F)/\Gamma$ is a homogeneous Finsler space
if and only if $\Gamma$ consists of Clifford translations of $(M, F)$.
In the process of the proof of the main theorem, we classify  all the Clifford
translations of symmetric Finsler spaces.


\end{abstract}

\section{Introduction}
\setcounter{equation}{0}

In this paper we give a sufficient and necessary condition for
a locally symmetric Finsler space to be globally homogeneous.
Specifically, we prove
\begin{thm}\label{maintheorem}
Let $\Gamma$ be a properly discontinuous group of isometries of a connected
simply connected globally symmetric Finsler space $(M, F)$.  Then
$(M, F)/\Gamma$ is a homogeneous Finsler space if and only if
$\Gamma$ consists of Clifford translations.  Further, if $(M, F)/\Gamma$ is
homogeneous, and if in the decomposition of $(M, F)$ as the
Berwald product of Minkowski space and irreducible symmetric Finsler
spaces, none of whose factors is \hfill\newline
\phantom{XXX} a compact Lie group with a bi-invariant Finsler metric,
\hfill\newline
\phantom{XXX} an odd-dimensional sphere with standard Riemannian metric,
\hfill\newline
\phantom{XXX}
a complex projective of odd dimension $>1$ with standard Riemannian metric,
\hfill\newline
\phantom{XXX} $\mathrm{SU}(2n)/\mathrm{Sp}(n)$, $n\geq 2$ with any
$\mathrm{SU}(2n)$-invariant Finsler metric
\newline\phantom{XXXXX}(there are non-Riemannian ones), nor
\hfill\newline
\phantom{XXX} $\mathrm{SU}(4n+2)/\mathrm{U}(2n+1)$,   $n\geq 1$, with any
$\mathrm{SU}(4n+2)$-invariant Finsler metric
\newline\phantom{XXXXX}(there are non-Riemannian ones),
\hfill\newline
then $(M, F)/\Gamma$ is symmetric.
\end{thm}

The Riemannian case of Theorem \ref{maintheorem} was formulated and
proved in the second author's papers \cite{WO60} for constant sectional
curvature and then \cite{WO62} in general.
\medskip

Recall that an isometry of a Riemannian manifold is called a Clifford translation if it moves each point of $M$ the same distance. This definition can be
generalized in an obvious way to the Finslerian case. Recently, the first author and M. Xu initiated the study of Clifford translations of Finsler spaces in a
series papers (\cite{DP, DM1, DM2, DM3}). They showed that many important results in the Riemannian case can be generalized to Finsler spaces. Moreover,
 some new phenomena were found in the Finslerian case. For example, there are some non-Riemannian Finsler spaces which are Clifford homogeneous, in the sense
that for any two points
$x_1, x_2$ in the manifold, there is a Clifford translation which maps $x_1$ to $x_2$. In view of that result it would be an interesting problem to classify all
the Clifford homogeneous Finsler spaces. We note that that the Riemannian Clifford homogeneous manifolds were classified by  Berestovskii and  Nikonorov
in \cite{BN081, BN082, BN09}.
Their list consists of the euclidean spaces, the odd-dimensional spheres with constant curvature, the connected simply connected compact simple Lie groups
with bi-invariant Riemannian metrics and the direct products of the above
manifolds.
\medskip

In Section 2, we present preliminaries of Finsler geometry. There we define
the notion of orthogonal  product of Finsler spaces and deduce a result on
the distance function of orthogonal product. It would be an interesting problem
to consider whether there is an analogue of the de Rham decomposition of
Finsler spaces, as in Riemannian geometry.
\medskip

The main results of this paper are obtained
through the study of Clifford translations of globally symmetric Finsler spaces.
In Section 3, the study is reduced to the case where $(M,F)$ is a Minkowski
space, or is of noncompact type, or is of compact type.  Clifford
translations of Minkowski spaces and symmetric Finsler spaces of
noncompact type are also treated in Section 3.  That leaves the case where
$(M,F)$ is of compact type.
\medskip

In Section 4 we reduce the proof for compact type to the case where $(M,F)$
is irreducible and of compact type.  This is one of the most delicate parts
of the proof.  In Section 5, the irreducibility allows
us to combine results of \cite{WO62} and \cite{OZ74} with a reduction to the
Riemannian case, completing the proof of Theorem \ref{maintheorem}.
\medskip

Although the main results of this paper are similar
to the Riemannian case, the arguments make use of some new ideas.
This generalization is important in Finsler geometry, where it is an
instance of the principle that the generalization of any important result
from Riemannian geometry to Finsler geometry may require a new viewpoint.

\section{Preliminaries}
\setcounter{equation}{0}
In this section, we will recall some definitions and notations in Finsler geometry. In particular,
we will survey some results on symmetric Finsler spaces.

\subsection{Finsler spaces}
\begin{defn}
 Let $V$ be a n-dimensional
real vector space. A Minkowski norm on $V$ is a real-valued function $F$ on
$V$  which is smooth on $V\backslash\{0\}$ and satisfies the conditions
\begin{description}
\item{\rm (1)}\quad $F(u)\geq 0$, $\forall u\in V$;
\item{\rm (2)}\quad $F(\lambda u)=\lambda F(u)$, $\forall\lambda >0$; and
\item{\rm (3)}\quad Given a basis $u_1, u_2, \cdots, u_n$ of $V$, write
$F(y)=F(y^1, y^2, \cdots, y^n)$ for
$y=y^1u_1+y^2u_2+\cdots+y^nu_n$. Then the Hessian matrix
$$(g_{ij}):=\left([\tfrac{1}{2}F^2]_{y^iy^j}\right)$$
is positive-definite at any point of $V\backslash\{0\}$.
\end{description}
The real vector space $V$ with the  Minkowski  norm $F$ is called a Minkowski space, usually denoted as $(V, F)$.
\end{defn}

It can be shown that  for a Minkowski norm $F$, we have
$F(u)>0$, $\forall u\ne 0$. Furthermore, we have  the triangle inequality:
 $$F(u_1+u_2)\leq F(u_1)+F(u_2),$$
 where the equality holds if and only if $u_2=\alpha u_1$ or $u_1=\alpha u_2$ for some $\alpha\geq 0$. From the triangle inequality one can deduce the
 fundamental identity:
 $$F(w)\geq w^iF_{y^i}(y),\,\,\mbox{for any}\,\, y\ne 0,$$
 with equality holding if and only if there is $\alpha\geq 0$ such that $w=\alpha y$; see \cite{BCS00}.

For a Minkowski norm $F$ on the real vector space $V$ we define:
$$C_{ijk}=\tfrac{1}{4}[F^2]_{y^iy^jy^k}.$$
Given $y\ne 0$, we can define two tensors on $V$, namely,
\begin{eqnarray*}
g_y(u, v)&=&\sum_{i,j=n}^ng_{ij}(y)u^iv^j,\\
C_y(u, v, w) &=&\sum_{i,j,k=1}^n C_{ijk}(y)u^iv^jw^k.
\end{eqnarray*}
They are called the fundamental tensor and the Cartan tensor,
respectively. Both the fundamental tensor and the Cartan tensor are symmetric. It is easily seen that
\begin{equation}\label{Cartan}
C_y(y, u, v)=0,\quad \mbox{for }\,\, y, u, v \in V \, \mbox{with } y \ne 0;
\end{equation}
see \cite{BCS00}.

\begin{defn}
 Let $M$ be a (connected)
smooth manifold. A Finsler metric on $M$ is a function $F$: $TM\to
[0,\infty)$ such that
\begin{description}
\item{\rm (1)}\quad $F$ is $C^\infty$ on the slit tangent bundle
$TM\setminus\{0\}$;
\item{\rm (2)}\quad The restriction of $F$ to any $T_xM$, $x\in M$ is a
Minkowski norm.
\end{description}
\end{defn}

Let $(M,F)$ be a Finsler space and $x, y\in M$. For any piecewise smooth
curve $\sigma (t)$, $0\leq~t\leq~1$ connecting $x$ and $y$, we
define the arc length of the curve by
$$L(\sigma)=\int_0^1 F(\sigma (t), \sigma'(t))dt.$$
The distance function $d$ of $(M,F)$ is
defined by
$$d(x,y)=\inf _{\sigma\in \Gamma (x,y)} L(\sigma),$$
where $\Gamma (x,y)$ denotes the set of all piecewise smooth
curves emanating from $x$ to $y$. It can be proved  that
$d(x, y)\geq 0 $ with the equality holding if and only if $ x=y$ and
$d(x,y)\leq d(x,z)+d(z,y)$, $\forall x,y,z\in M$. However,
generically we cannot have $d(x,y)=d(y,x)$. Therefore $d$ is not a
metric space distance in the usual sense.

\subsection{The Chern connection}

Let $(M,F)$ be a Finsler space and $(x^1, x^2, \cdots, x^n)$  a
local coordinate system on an open subset  $U$ of $M$. Then
$\frac{\partial}{\partial x^1},\cdots, \frac{\partial}{\partial
x^n}$ form a basis for the tangent space at any point in $U$.
Therefore we have the coefficients $g_{ij}$ and  $C_{ijk}$. Define
$$C^i_{jk}=g^{is}C_{sjk}.$$
The formal Christoffel symbols of the second kind are
$$\gamma^i_{jk}=g^{is}\frac{1}{2}\left(\frac{\partial g_{sj}}{\partial
x^k}-\frac{\partial g_{jk}}{\partial x^s}+\frac{\partial
g_{ks}}{\partial x^j}\right).$$ They are functions on $TU\setminus\{0\}$.
We can also define some other quantities on $TU\setminus\{0\}$ by
$$N^i_j (x,y):=\gamma^i_{jk}y^k -C^i_{jk}\gamma^k_{rs}y^r y^s,$$
where $y=\sum y^i\frac{\partial}{\partial x^i}\in T_x (M)\setminus\{0\}$.

Now the slit tangent bundle $TM\setminus\{0\}$ is a fiber bundle  over the
manifold $M$ with the natural projection $\pi$. Since $TM$ is a
vector bundle over $M$, we have a pull-back bundle $\pi^* TM$ over
$TM\setminus\{0\}$.

\begin{thm}
 \textbf{\rm (Chern \cite{CH43})}\quad  The pull-back
bundle $\pi^*TM$ admits a unique linear connection,  which is torsion free and almost $g$-compatible.
The coefficients of the connection is
$$\Gamma^l_{jk}=\gamma^l_{jk}-g^{li}\tfrac{1}{F}(A_{ijs}N_k^s-A_{jks}N^s_i
+A_{kis}N^s_j).$$
\end{thm}

In the literature, the above connection is called the Chern connection. Using the Chern connection, we can define the notions of geodesics, exponential map for Finsler spaces as in Riemannian geometry;
see \cite{BCS00} for the details.  Note that the Chern connection is invariant
under isometries of $(M,F)$.

\begin{defn}
A Finsler space $(M, F)$ is called a Berwald space if in a local standard coordinate system the coefficients $\Gamma^{jk}_i$ are functions of $x\in M$ only.
In this case, the coefficients $\Gamma$ define an affine  connection on the underlying  manifold $M$.
\end{defn}

\begin{rem}\label{szabo}{\rm
It was proved by Z. I. Szab\'{o} (\cite{SZ81})
that, if $(M, F)$ is a Berwald space, then there exists a Riemannian metric $Q$ on $M$ whose Levi-Civit\`a connection coincides
with the linear connection of $(M, F)$.  Then of course $(M,Q)$ is unique up to
affine diffeomorphism.}
\end{rem}

\subsection{Symmetric Finsler spaces}
\begin{defn}
A  Finsler space $(M, F)$ is called locally symmetric if for any $x\in M$, there exists a neighborhood $U$ of $x$ such that
the geodesic symmetry on $U$ is a local isometry. It is called globally symmetric if  for any $x\in M$, there is an involutive isometry $\rho$ with $x$
as an isolated fixed point.
\end{defn}

It is the main result of \cite{DH07} that any globally symmetric Finsler space (or a complete locally symmetric Finsler space) must be of the Berwald type.
Therefore we usually consider the more generalized class of globally or locally affine symmetric Berwald spaces. Recall that a Berwald space $(M, F)$ is
called globally (resp. locally) affine symmetric if its connection is globally (resp. locally) affine symmetric. It is easily seen that a reversible globally
affine symmetric Berwald space  must be globally symmetric. However, there are many symmetric Berwald spaces that are not reversible.

To study affine symmetric Berwald spaces, we introduced the notion of a
Min-kowski Lie algebra in \cite{DH07}. From now on, we shall simplify the
term ``globally (resp. locally) affine symmetric spaces'' as ``GASBS
(resp. LASBS)''.

\begin{defn}
Let $ ({\mathfrak g},\sigma)$
be a symmetric Lie algebra and let ${\mathfrak g}={\mathfrak h}+{\mathfrak m}$
denote
the canonical decomposition of ${\mathfrak g}$ with respect to the
involution $\sigma$. Let $F$ be a Minkowski norm on ${\mathfrak m}$
such that
$$g_y([x,u],v)+g_y(u,[x,v])+2C_y([x,y],u,v)=0,$$
for all $y,u,v\in {\mathfrak m}$ and $x\in {\mathfrak h}$ with $y \ne 0$,
where $g_y$ and $C_y$
are the fundamental form and Cartan tensor of $F$ at $y$, respectively.
Then $({\mathfrak g}, \sigma, F)$ is called a Minkowski symmetric Lie algebra.
\end{defn}

There exists a correspondence between affine symmetric Berwald spaces and
Minkowski symmetric Lie algebras, similar to the orthogonal
involutive Lie algebra correspondence of the
Riemannian case. Each Minkowski Lie algebra must be an orthogonal symmetric Lie algebra with  respect to some inner product. Hence any Minkowski symmetric
Lie algebra can be decomposed into the direct sum of  an abelian ideal, a Minkowski symmetric Lie algebra of compact type and a Minkowski symmetric Lie
algebra of noncompact type. Furthermore, a Minkowski symmetric Lie algebra of compact or noncompact type can be decomposed into the direct sum of
irreducible ones (see \cite{DH07}).  From this we deduce the following.

\begin{thm}\label{berwaldproduct}
Let $(M, F)$ be a connected simply
connected GASBS. Then $(M, F)$ can be
 decomposed into the product of a  Minkowski
space, a GASBS of compact type and a
GASBS of noncompact type. Moreover,
every simply connected GASBS of compact or
noncompact type can be  decomposed  into the product of
irreducible GASBS's. The decomposition is
unique as manifolds but in general is not unique as Finsler spaces.
\end{thm}
 In the following, we usually say that $(M, F)$ is the Berwald product of
a Minkowski space and the irreducible GASBS's.

 \section{Product decompositions of Clifford translations}\label{sec3}
\setcounter{equation}{0}

In this section and the next, we see how product decompositions of Theorem
\ref{berwaldproduct} leads to a corresponding product decomposition
of Clifford translations.  This is much more
delicate than the comparable decompositions in the Riemannian case.
Here in Section \ref{sec3} we develop the basics of these decompositions and
their specialization to Minkowski spaces and Finsler symmetric spaces
of noncompact type.  In Section \ref{sec4} we take a close look at the
situation for Finsler symmetric spaces of compact type.  Then in Section
\ref{sec5} we will use these product decompositions to complete the proof of
Theorem \ref{maintheorem}.
\medskip

We first adjust some definitions from \cite{OZ74} to fit the Finsler framework.

\begin{defn} \label{def-pres}
Let $\sigma$ be a map from a smooth manifold $M$ into itself and $\gamma: {\mathbb R}\to M$ a curve in $M$. We say that $\sigma$ preserves $\gamma$ if there
is a constant $c\in {\mathbb R}$ such that $\sigma (\gamma (t))=\gamma (t+c)$, for all $t\in {\mathbb R}$.
\end{defn}

The following lemma extends a result on Riemannian manifolds;  see \cite{OZ74}.
\begin{lem}
Let $(M, F)$ be a Finsler space, $\sigma$ an isometry of $(M, F)$,
$\gamma: {\mathbb R}\to M$ a geodesic of constant speed, and
$c\in {\mathbb R}$\, a constant. The following are equivalent:
\begin{description}
\item{\rm (1)}\quad $\sigma(\gamma (t))=\gamma (t+c)$ for all $t$,
in other words $\sigma$ preserves $\gamma$;
\item{\rm (2)}\quad $\sigma_*(\dot{\gamma}(t))=\dot{\gamma}(t+c)$ for all $t$;
\item{\rm (3)}\quad $\sigma_*(\dot{\gamma}(t)^\perp)=\dot{\gamma}(t+a)^\perp$
and $g_T(\sigma_*(\dot{\gamma}(t)), \dot{\gamma}(t+c))\geq 0$,
where $g_T$ is the fundamental tensor,   $T$ is the tangent vector field of
$\gamma$, and $\perp$ is taken with respect to $g_T$.
\end{description}
\end{lem}

For the proof, note that a Finsler space isometry sends geodesics to geodesics.

\begin{thm}\label{oz}
Let $\sigma$ be an isometry of a complete Finsler space $(M, F)$. Then
the following are equivalent:
\begin{description}
\item{\rm (1)}\quad $\sigma$ is a Clifford translation of $(M,F)$;
\item{\rm (2)}\quad If $x\in M$ there is a minimal geodesic $\gamma$ from $x$ to $\gamma(x)$ preserved by $\sigma$;
\item{\rm (3)}\quad If $x\in M$ and $\gamma$ is a minimal geodesic
from $x$ to $\sigma (x)$, then $\sigma$ preserves $\gamma$.
\end{description}
\end{thm}

The proof of this theorem is similar to Ozols' proof for
the Riemannian case in \cite{OZ74}. Here one just need use $g_T$, where
$T$ is the tangent vector
field of the curves,  to  replace the inner product, and use the first variation formula in Finsler geometry (see \cite{BCS00}, page 123).
We omit the details here; see \cite{OZ74}.
\medskip

Now we turn to the product decompositions.
We start by  defining a new notion of product for Finsler spaces.
\begin{defn}\label{defn3.1}
Let $(M, F)$ and $(M_i, F_i)$, $i=1,\dots,s$ be  Finsler spaces. We say that $(M, F)$ is the orthogonal product of $(M_i, F_i)$ if the following conditions
are satisfied:
\begin{description}
\item{\rm (1)}\quad The manifold $M$ is the product of the manifolds $M_i$:
$M=M_1\times M_2\times \cdots\times M_m$\,, and the metrics satisfy
$F_i(y)=F(y)$\, for all $i$ and all $y\in T(M_i)\setminus \{0\}$.
\item{\rm (2)}\quad Let $p_i$ be the projection of $M$ onto $M_i$. Then a smooth curve $\gamma(t)$, $a<t<b$, is a geodesic of $(M, F)$ if and only if the
curve $p_i(\gamma)$ is a geodesic of $(M_i, F_i)$, for all $i$.
\item{\rm (3)}\quad Let $g$ be the fundamental tensor of $M$ and
$x = (x_1, x_2, \cdots,x_s) \in M$.  If $y\in T_{x_i}(M_i)\setminus\{0\}$ and
$v\in T_{x_j}(M_j)$ with $i\ne j$ then $g_y(y, v)=0$.
\end{description}
\end{defn}
Notice that geodesics under our consideration are always of constant speed.
The product of Riemannian manifolds is a typical example of an orthogonal product. We will prove in the following that
 the Berwald product of irreducible globally affine Berwald spaces is orthogonal. Notice that the orthogonal product of the given Finsler spaces is
  generically not unique, as one can see from the Berwald product of GASBS's
(see \cite{DH07}).  We now start to construct these product decompositions.

\begin{lem}\label{lem3.1}
Let $(M, F)$ be the orthogonal product of $(M_i, F_i)$, $i=1,\dots,s$. Then
 $F(v_1+v_2+\cdots v_s)\geq F(v_1)$, for any $v_1\in T(M_1)\setminus \{0\}$ and $v_k\in T(M_j)$, $k=2,\dots,s$. The equality holds if and
 only if $v_k=0$, for $k=2,\dots,s$.
\end{lem}
\begin{proof}
Fix $x=(x_1,\dots,x_s)\in M$.  Choose a local coordinate system
\begin{equation}\label{coordinate}
(x_1^1,\dots, x_1^{n_1},x_2^1,\dots, x_2^{n_2},\dots, x_s^{1},\dots, x_s^{n_s})
\end{equation}
 of $M$ around $x$ such that
$(x_i^1,\dots, x_i^{n_i})$
 is a local coordinate system of $M_i$ on a neighborhood of $x_i$. For the convenience of the following we relabel the coordinate system (\ref{coordinate})
 as
 $$(z^1, z^2,\cdots, z^{n_1+n_2+\dots n_s}).$$
 Denote
$w=v_1+v_2+\dots + v_s$ and $w_1=v_2+\dots + v_s$. By the fundamental equality
we have $$F(w)\geq w^j[F_{y^j}(y)]|_{y=v^1},$$
where $w^j$ is determined  by 
$$w=\sum_{j=1}^{n_1+\dots+ n_s} w^j\frac{\partial}{\partial z^j},$$
and  the automated summation is taken over the range from $1$ to $n_1+\dots n_s$.
Moreover, the equality holds if and only if there is $\alpha\geq 0$ such that $w=\alpha v_1$, in other words if and only if $v_2=\dots v_s=0$. From this we get that
$$F(w)\geq (v_1)^j[F_{y^j}(y)]|_{y=v^1}+(w_1)^j[F_{y^j}(y)]|_{y=v^1}=F(v_1)+(w_1)^j [F_{y^j}(y)]|_{y=v_1},$$
where we have used the Euler's theorem on homogeneous functions (see \cite{BCS00}). Now using again the Euler's theorem,
we have $F_{y^iy^j}(y)y^i=0$. Thus
\begin{eqnarray*}
g_{v_1}(v_1, w_1)&=&g_{ij}(v_1)(v_1)^i(w_1)^j\\
&=&[F_{y^iy^j}(y)F(y)]|_{y=v_1}(v_1)^i(w_1)^j+[F_{(y^i}F_{y^j}]|_{y=v_1}(v_1)^i(w_1)^j\\
&=& F(v_1)[F_{y^iy^j}]|_{y=v_1}(v_1)^i(w_1)^j+[F_{y^i}(y)]|_{y=v_1}(v_1)^i[F_{y^j}(y)]|_{y=v_1}(w_1)^j\\
&=& F(v_1)(w_1)^j[F_{y^j}(y)]|_{y=v_1}.
\end{eqnarray*}
Since $(M, F)$ is an orthogonal product of  $(M_i, F_i)$ we have
$g_{v_1}(v_1, w_1)=0$. Since $F(v_1)\ne 0$ the assertion follows.
\hfill \end{proof}

\begin{lem}\label{lem3.2}
Let $(M, F)$ be a complete Finsler space which is the orthogonal product
of the Finsler spaces $(M_i, F_i)$, $i=1,\dots,s$.
For any $x_j\in M_j$, and $x_j'\in M_j$, $j=1,2,\dots, s$,  we have
\begin{equation}\label{distanceeq}
d((x_1,x_2,\dots, x_s), (x_1', x_2',\dots, x_s'))\geq d_1(x_1, x_1'),
\end{equation}
where $d, d_1$ are the distance functions of $(M, F)$ and $(M_1, F_1)$, respectively. The equality holds if and only if $x_j'=x_j$ for $j=2,\dots, s$.
\end{lem}
\begin{proof}
For any curve $\gamma$ in $M$ connecting $x=(x_1,\dots, x_s)$ and $x'=(x_1',\dots,x_s')$,  $p_1(\gamma)$ is a curve in $M_1$ connecting $x_1$ and $x_1'$.
By Lemma \ref{lem3.1}, the arc length of $p(\gamma)$ is less than or equal to $\gamma$. Taking the infimum we prove the inequality (\ref{distanceeq}).
Now we prove the second assertion.
Select a minimal geodesic $\gamma_1(t): a<t<b$ in $M_1$ connecting $x_1$ and $x_1'$. Then by the condition (2) of Definition \ref{defn3.1}, the curve
$t\to (\gamma_1(t), x_2, \dots, x_s)$ is a geodesic of $(M, F)$, which has   the same length of $\gamma$. Therefore we have
$$d((x_1,x_2,\dots, x_s), (x_1', x_2,\cdots, x_s))\leq d_1(x_1, x_1').$$
This proves the ``if'' part of the assertion.
On the other hand, suppose there is one $x_k$ which is not equal to $x_k'$, say $x_2\ne x_2'$.
Let $\zeta(t), a\leq t\leq b$ be a minimal geodesic in $M$ connecting $(x_1,x_2,\dots, x_s)$ and $(x_1', x_2',\cdots, x_s')$. Then $p_1(\zeta)$ is a
geodesic in $M_1$ connecting $x_1$ and $x_1'$ and $p_j(\zeta)$ is a  geodesic in $M_k$ connecting $x_j$ and  $x_j'$. By the assumption,
the tangent vector of $p_2(\zeta)$ is everywhere nonzero. Then by Lemma \ref{lem3.1}, the length of $\zeta$, which is just
$$d((x_1, \dots, x_s), (x_1',\dots, x_s')),$$
  is larger than that of $p_1(\zeta)$. Therefore we have
  $$d((x_1, \dots, x_s), (x_1',\dots, x_s'))> d_1(x_1, x_1').$$
  This completes the proof of the lemma.
\hfill \end{proof}
\medskip

We also need a result on conjugate loci and cut loci of GASBS's. The notions  and the fundamental properties of conjugate points and cut points of
a general Finsler space can be found in \cite{BCS00}. Suppose $(M, F)$ is a Berwald space and $x\in M$. As the Chern connection of $(M, F)$ is an
affine  connection, we can define the notion of conjugate points of $x$ in exactly the same way as in Riemannian case, through the exponential map.
On the other hand, suppose $x\in M$ and $\gamma$ is a geodesic emanating from $x$. A point $p$ along $\gamma$ is called a cut point along $\gamma$
if $\gamma$ minimizes arc length up to $p$ but no further. Accordingly, one can define the notions of the  conjugate loci and cut loci, as well as the first
conjugate locus.

\begin{lem}\label{cutloci}
Let $(M, F)$ be a connected and simply connected GASBS. Then for any $x\in M$, the first conjugate locus of $x$ coincides with the  cut locus of $x$.
\end{lem}
\begin{proof}
The proof for the Riemannian case of this result was provided by Crittenden  in \cite{CR61}. As in Remark \ref{szabo}, given a GASBS $(M, F)$ there is a
Riemannian metric $Q$ on $M$ whose Levi-Civit\`a connection coincides with
the Chern connection of $(M, F)$. Then $(M, Q)$ is a Riemannian symmetric
space. Hence the
conjugate points of $x$ in $(M, F)$ and in $(M, Q)$ are the same sets. Based on this observation and Corollary 8.2.2 in \cite{BCS00}, which asserts that
if $x'$ is the cut point to $x$ along a geodesic $\gamma$, then either $x'$ is the first conjugate point along $\gamma$ or there exists at least two distinct
geodesics with the same arc length from $x$ to $x'$, the proof of  Theorem 4 of  \cite{CR61} applies also to the compact Berwaldian case without any change.
Thus the lemma is true when the full group of isometries of $(M, F)$ is compact semi-simple. For the general case,  notice  the following facts:
\begin{description}
\item{(1)}\quad The globally affine symmetric space $(M, F)$ can be decomposed into the product
$$(M, F)=(M_0, F_0)\times (M_1, F_1)\times(M_2, F_2).$$
where $(M_0, F_0)$ is a Minkowski space and $(M_1, F_1)$ (resp. $(M_2, F_2)$) is  the products of non-compact (resp. compact) irreducible globally
affine symmetric Berwald spaces.
\item{(2)}\quad The Minkowski space $(M_0, F_0)$ contributes nothing either
to the first conjugate loci or cut loci of $M$.
\item{(3)}\quad It is proved in \cite{DH05} that the non-compact
$(M_1, F_1)$ has flag curvature everywhere $\leq 0$. Hence by the Cartan-Hadamard theorem (\cite{BCS00}), $(M_1, F_1)$ contributes nothing to either
the first conjugate loci or the cut loci of $(M, F)$.
\end{description}
Now if $x'=(x_0', x_1', x_2')$ is the first conjugate point of $x=(x_0, x_1,x_2)$ along a geodesic $\gamma$. Then we have $x_0'=x_0$ and $x_1'=x_1$.
Thus $x_2'$ must be the first conjugate point to $x_2$ in $(M_2, F_2)$ along $p_2(\gamma)$. By the above argument, $x_2'$ must be the cut point of $x_2$
along $p_2(\gamma)$ (in $(M_2, F_2)$).
By Lemma \ref{lem3.2},  $x'$
must be the cut point of $x$ along $\gamma$. The converse assertion is obvious.
\hfill \end{proof}
\medskip

Here is the next step in reducing the proof of Theorem \ref{maintheorem} to
the flat and the irreducible cases.

\begin{thm}\label{thm3.2}
Let $(M, F)$ be a connected simply connected GASBS and suppose
that $(M, F)$ is decomposed into the Berwald product
$$(M, F)=(M_0, F_0)\times (M_1, F_1)\times (M_2, F_2),$$
where $(M_0, F_0)$ is a Minkowski space and $(M_1, F_1)$ (resp. $(M_2, F_2)$)
is  a GASBS  of noncompact type (resp. compact type). Suppose
$\sigma$ is a Clifford translation of $(M, F)$. Then there are Clifford
translations of $\sigma_i$ of $(M_i, F_i)$, $i=0, 1,2$, such that
$$\sigma=\sigma_0\times \sigma_1\times \sigma_2.$$
Moreover, $\sigma_0$ must be an ordinary translation of the affine space
underlying $M_0$\,, and $\sigma_1$ must be trivial.
\end{thm}

\begin{proof}
Suppose $\sigma$ is a Clifford translation of $(M, F)$. Then as an isometry of $(M, F)$,
$\sigma$ must keep each of the  $M_0$, $M_1$ and $M_2$ invariant. Thus
there are isometries $\sigma_i$ of $M_i$, $i=0,1,2$,  such that
$\sigma=\sigma_0\times \sigma_1\times \sigma_2$.
We now prove that $\sigma_i$, $i=0,1,2$,  is a Clifford translation of $(M_i, F_i)$. Let $x_i\in M_i$, $i=0,1,2$,  and suppose
$\gamma$ is a minimal geodesic of unit speed in $(M, F)$ connecting $x=(x_0, x_1, x_2)$ and $\sigma(x)$. We assert that $\gamma_i=p_i(\gamma)$, $i=0, 1,2$,
is a minimal geodesic of $(M_i, F_i)$ connecting $x_i$ and $\sigma _i(x_i)$. For $i=0$ or $1$ this is obvious, since any geodesic must be minimal in $M_0$ or $M_1$.
Suppose conversely that $p_2(\gamma)$ is not a minimal geodesic between $x_2$ and $\sigma_2(x_2)$. Suppose $\sigma_2(x_2)=\gamma_2(t_0)$, $t_0>0$. Then there
is $t_1\in (0, t_0)$ such that $\gamma_2(t_1)$ is a cut point of $x_2$ along $\gamma_2$. By Lemma \ref{cutloci}, $\gamma_2(t_1)$ must be the first conjugate point
of $x_2$ along the geodesic $\gamma_2$. But then $\gamma(t_1)=(\gamma_0(t_1), \gamma_1(t_1), \gamma_2(t_1))$ must be a conjugate point of $x$ along the
geodesic $\gamma$, since the connection of $(M, F)$ is the product of that of the $(M_i, F_i)$, $i=0,1,2$. This implies that $\gamma$ cannot be minimal
between $x$ and $\gamma (t)$, for any $t>t_1$. That is a contradiction. Hence $\gamma_i$ are all minimal between $x_i$ and $\gamma_i(x_i)$, $i=0,1,2$. Now
by Lemma \ref{oz}, there is a constant $a\in {\mathbb R}$ such that $\sigma (\gamma(t))=\gamma (t+a)$. Then $\sigma_i(\gamma_i(t))=\gamma_i(t+a)$, $i=0,1,2$.
Using Lemma \ref{oz} again, we conclude that $\sigma_i$ is a Clifford translation of $(M_i, F_i)$, for $i=0,1,2$.
\smallskip

Suppose $G$ is the full group of isometries  of the Minkowski space
$(M_0, F_0)$. Let $L$ be the isotropy subgroup of $G$ at the origin $0$.
Then by Theorem 3.1 of \cite{DH04}, there exists $G$-invariant Riemannian
metric $Q$ on $M_0$. Since $G$ must contain all the parallel translations of
the linear
vector space, $Q$ is a euclidean metric and $G$ is the semidirect product of
its translation subgroup with $L$.  It follows (see \cite{WO60}) that a
Clifford translation of $(M_0, F_0)$ must be an ordinary translation.
\smallskip

Now we consider $(M_1, F_1)$, which is a GASBS  of noncompact type. If the
Clifford translation $\sigma_1$ is nontrivial, then for any $x_1\in M_1$
there is a unique geodesic $\zeta_{x_1}$ of unit speed connecting $x_1$ and
$\sigma_1(x_1)$. Denote the initial tangent
vector of $\zeta_{x_1}$ by $X_{x_1}$. We obtain a smooth vector field of $X$ on $(M_1, F_1)$ which has length $1$ everywhere. By the above mentioned result
of Szab\'{o}, there is a Riemannian metric $Q_1$
on $M_1$ whose Levi-Civit\`a connection coincides with the Chern connection of $(M_1, F_1)$. But then $(M_1, Q_1)$ is a Riemannian symmetric space which is
the product of irreducible Riemannian symmetric spaces of the non-compact type. Since both $F_1$ and $Q_1$ are homogeneous metrics on $M_1$, the assumption
that the vector field $X$ has constant length $1$ with respect to $F_1$ implies that it has bounded length with respect to $Q_1$. Now by the main theorem of the
second author's article \cite{WO62}, $X$ must be a nonzero parallel vector
field with respect to the Levi-Civit\`a connection of $Q_1$. This is
impossible. Hence any Clifford translation of $(M_1, F_1)$ must be trivial.
This completes the proof of the theorem. \hfill
\end{proof}

\section{Clifford translations of compact Finsler symmetric spaces}
\label{sec4}
\setcounter{equation}{0}
Now we study Clifford translations of connected simply connected
GASBS's of compact type, completing the considerations of Theorem
\ref{thm3.2} and then reducing the analysis of Clifford translations to the
case of irreducible compact Riemannian symmetric spaces.
\smallskip

Recall that if $x$ is a point in a compact Finsler manifold $(M, Q)$, then
a point $x'$ is called an antipodal point to $x$ if it is of maximal distance
from $x$. We denote the set of antipodal
points of $x$ by $A_x$ and call it the  antipodal set of $x$.

\begin{thm}
Let $(M, F)$ be a GASBS whose full group of isometries is compact and
semisimple. Suppose $\sigma$ is a
Clifford translation of $(M, F)$ and $(M, F)$ has a decomposition
\begin{equation}\label{gasbs-compact}
(M, F)=(M_1, F_1)\times \dots \times (M_s, F_s)
\end{equation}
where each $(M_j, F_j)$ is  an irreducible GASBS.
Then {\rm (\ref{gasbs-compact})} is an orthogonal product of Finsler
spaces, and $\sigma$ decomposes as
$$\sigma=\sigma_1\times\dots\times \sigma_s$$
where $\sigma_j$ is a Clifford translation of $(M_j, F_j)$ for $j=1,\dots, s$.
\end{thm}

\begin{proof}
We first prove that (\ref{gasbs-compact}) is an orthogonal product of Finsler
spaces.
The condition (1) of Definition \ref{defn3.1} is obviously satisfied. Since
the affine connection of $M$ is the product of the  affine connections of
$(M_j, F_j)$, (2) is also satisfied. Now we proceed to prove (3).
Let ${\mathfrak g}$ and ${\mathfrak g}_j$  be the respective Lie algebras of
identity component of the full groups of isometries of the Berwald spaces
$(M, F)$ and $(M_j, F_j)$. Suppose that
${\mathfrak g}={\mathfrak k}+{\mathfrak p}$ and
${\mathfrak g}_j={\mathfrak k}_j+{\mathfrak p}_j$ are the canonical
decompositions of ${\mathfrak g}$ and ${\mathfrak g}_j$. Then with respect
to those decompositions, $({\mathfrak g}, F)$ and $({\mathfrak g}_j, F_j)$
are Minkowski symmetric Lie algebras. Fixing  a designated
 origin $x=(x_1, x_2,\dots, x_s)$ of $M$, we identify the tangent space
$T_x(M)$ with ${\mathfrak p}$, and each $T_{x_j}(M_j)$ with ${\mathfrak p}_j$.
Now suppose $y\in {\mathfrak p}_\ell\setminus \{0\}$ and
$v\in {\mathfrak p}_j\setminus\{0\}$  with $\ell\ne j$. By the definition of a
Minkowski Lie algebra, we have (taking $u=y$)
$$g_y([w, y], v)+g_y(y, [w, v])+2C_y([w, y], y, v)=0$$
for any $w\in {\mathfrak k}$ and $v\in {\mathfrak p}_j$. By (\ref{Cartan}),  the third term in the left side of the equation is equal to $0$. Moreover,
$[w, y]=0$, for any $w\in \bigoplus_{i\ne \ell} {\mathfrak k}_i$. Thus
we have $g_y(y, [w, v])=0$, for any $w\in \bigoplus_{i\ne \ell} {\mathfrak k}_i$ and
$v\in {\mathfrak p}_j$. Since $({\mathfrak g}_j, F_j)$ is an irreducible
Minkowski symmetric Lie algebra, we have
$[{\mathfrak k}_j, {\mathfrak p}_j]={\mathfrak p}_j$. From this we deduce
that  $g_y(y, v)=0$, for any $v\in {\mathfrak p}_j$. This completes the
argument that (\ref{gasbs-compact}) is an orthogonal product of Finsler
spaces.
\smallskip

 Now the action of $\sigma$ can be written as
 $$\sigma(x_1,x_2,\dots, x_s)=(f_1(x_{\tau(1)}), f_2(x_{\tau(2)}),\dots, f_s(x_{\tau(s)}),$$
 where $\tau$ is a permutation of $(1,2,\dots,s)$ and $f_j$ is an isometry from $M_{\tau(j)}$ onto $M_j$. Using a similar argument about the minimal geodesics
 as above, we can prove that if $\tau$ is the trivial permutation, then each $f_j$ is a Clifford translation of $(M_j, F_j)$. Next we consider the case
 that $\tau$ is nontrivial. Then $\tau$ is the product of disjoint cycles. To obtain a complete understanding, we first consider the case that
 $\tau=(1,2,\dots,s)$. Then the Berwald spaces $(M_i, F_i)$ are isometric to each other. Thus we can identify each $(M_j, F_j)$ with $(M_1, F_1)$ and suppose
 $$\sigma (x_1,x_2,\dots, x_s)=(f_1(x_s), x_1,\dots, x_{s-1}),\quad x_i\in M_1.$$
 As $\sigma$ is Clifford,
 $$d(x, \sigma (x))=d((x_1, x_2,\dots, x_s), (f_1(x_1), x_1,\dots, x_{s-1}))$$
 is a constant $c>0$. The choice $x=(x_1, x_1,\dots, x_1)$ would give us
 $$c=d((x_1, x_1,\dots,x_1),(f_1(x_1),x_1,\dots, x_1)).$$
 By Lemma \ref{lem3.2}, we get that $d_1(x_1, f_1(x_1))=c$, where $d_1$ is the distance function of $(M_1, F_1)$. On the other hand, if $s>2$, then  we can
 take $x_2\ne x_1=x_s$ in $M_1$. Then we also have
 $$c=d((x_1, x_2,\dots, x_s), (f_1(x_1), x_1,\dots, x_{s-1}))=d_1(x_1, f_1(x_1)).$$
 This is a contradiction with Lemma \ref{lem3.2}. Therefore $s\leq 2$.

 If $s=2$, considering the point $(x_1, x_1)$ we get
 $$c=d((x_1, x_1), (f_1(x_1), x_1))=d_1(x_1, f_1(x_1)).$$
On the other hand, suppose $x_2$ is an antipodal point in $M_1$ of maximal distance from $x_1$  and $l=d_1(x_1, x_2)$. Then we also have
$$c=d((x_2, x_1), (f_1(x_2), x_2))\geq d_1(x_1, x_2)=l.$$
This implies that  $c=l$. Moreover, by Lemma \ref{lem3.2}, we must also have $x_1=f_1(x_2)$. This means that each point $x_1$ of $(M_1, F_1)$
has a unique antipodal point
and $f_1(x_1)$ is exactly the antipodal point of $x_1$. In the following, we denote the (unique) antipodal point of $x\in M_1$ by $A_{x}$.
Then $f_1(x)=A_x$, for any $x\in M_1$. In particular, $f_1$ is a Clifford translation of $(M_1, F_1)$.
\smallskip

In summary, we have shown that $(M, F)$ has a $\sigma$--invariant
decomposition
$$(M, F)=(M_1, F_1)\times \dots\times (M_k, F_k)\times
	(M_1', F_1')\times \dots \times (M_t', F_t'),$$
 where $(M_j, F_j)$ is an irreducible GASBS for $j=1, 2,\dots,k$
and $(M_j', F_j')$ is a  Berwald product of two copies of  an irreducible GASBS
$(N_j, L_j)$ for  $j=1,2,\dots, t$.  Here $s = k+2t$, but $k$ or $t$ may be $0$.
The point is that  $\sigma$ decomposes as
 $$\sigma=\sigma_1\times\dots\times \sigma_k\times \sigma_1'\times\dots\times \sigma_t',$$
 where $\sigma_j$ is a Clifford translation of $(M_j, F_j)$ for
$j=1,\dots, k$, and $\sigma_j'$  is a Clifford translation of $(M_j', F_j')$
for $j = 1,\dots , t$.
\smallskip

Suppose $t\ne 0$.
  Let $Q_j$ be a Riemannian metric on $N_j$ whose Levi-Civit\`{a} connection coincides the Chern connection of $L_j$. Using a similar argument on minimal geodesics as above, one can easily show that $\sigma'_j$ is also a Clifford translation of the Riemannian product $(N_j, Q_j)\times (N_j, Q_j)$. Moreover, by the arguments above, $\sigma_j'$ must be
 of the form $\sigma_j'(x_1, x_2)=(A_{x_2}, x_1)$, where $A_{x_2}$ is the unique antipodal point of $x_2$ with respect to $Q_j$. Now we  fix $x_1$ and let $\gamma(t)$, $0\leq t\leq l_j$  be a minimal geodesic of $(N_, Q_j)$ (with unit speed) connecting
  $x_1$ and $A_{x_1}$.   Since
  $$d_j(x_1, A_{\gamma(t)})=d_j(A_{x_1}, A_{A_{\gamma(t)}})=d_j(A_{x_1}, \gamma(t))=l_j-t,$$
  where $d_j$ denotes the distance of $(N_j, Q_j)$, we get that, with respect to the distance $d_j'$ of $(N_j, Q_j)\times (N_j, Q_j)$,
  $$d_j'((x_1, \gamma(t)), \sigma_j'(x_1,\gamma(t)))=\sqrt{d_j^2(x_1, A_{\gamma(t)})+d^2_j(\gamma(t), x_1)}=\sqrt{(l_j-t)^2+t^2},$$
which is not a constant for $0\leq t\leq l_j$. Therefore $\sigma_j'$
is not a Clifford translation. This is a contradiction. Hence $t$ must be zero and the proof is completed.
\hfill \end{proof}

\section{Proof of Theorem $\mathbf{1.1}$} \label{sec5}
\setcounter{equation}{0}
In this section we complete the proof of Theorem \ref{maintheorem} by
reducing it to the known Riemannian case \cite{WO62}.  The main tools
for that reduction are the Szab\' o construction of Remark \ref{szabo} and
our product decompositions of Sections \ref{sec3} and \ref{sec4}.  We also
need a result from Kobayashi and Nomizu \cite[Chapter VI, Theorem 3.5]{KN63}:

\begin{prop}\label{KN}
Let $(M,Q)$ be a complete irreducible Riemannian manifold with $\dim M > 1$.
Let $\mathfrak I(M,Q)$ denote the group of all isometries of $(M,Q)$, and
$\mathfrak A(M,Q)$ the group of all
affine diffeomorphisms of $(M,Q)$.  Then $\mathfrak A(M,Q) = \mathfrak I(M,Q)$.
\end{prop}

This gives a bridge between Berwald spaces and Riemannian manifolds:

\begin{cor}\label{inj-isom}
Let $(M,F)$ be a complete irreducible Berwald space, $\dim M > 1$.  Let
$\mathfrak I(M,F)$ denote its isometry group.  Let $(M,Q)$ be a Riemannian
manifold whose Levi-Civit\` a connection coincides with the linear
connection of $(M,F)$, as in {\rm Remark \ref{szabo}}.  Then
$\mathfrak I(M,F)$ is a closed subgroup of $\mathfrak I(M,Q)$.
\end{cor}
\begin{proof} Since an isometry of a Berwald space preserves its Chern
connection, Proposition \ref{KN} gives us $\mathfrak I(M,F) \subset
\mathfrak I(M,Q)$.  These Lie groups carry the compact open
topology from their action on $M$, so $\mathfrak I(M,F)$ is closed
in $\mathfrak I(M,Q)$.
\hfill \end{proof}
\vskip 2pt

The result is stronger in the symmetric case:

\begin{thm}\label{equal-isom}
Let $(M,F)$ be a connected, simply connected irreducible GASBS of compact
type, and
$(M,Q)$ a Riemannian manifold whose Levi-Civit\` a connection coincides
with the linear connection of $(M,F)$, as in {\rm Remark \ref{szabo}}.
Then $\mathfrak I(M,F)$ is a subgroup of finite index in $\mathfrak I(M,Q)$
which contains the identity component.
\end{thm}
\begin{proof}
Let $G$ denote the subgroup of $\mathfrak I(M,F)$ generated by products
of pairs of symmetries, in other words by transvections.  Since
$(M,Q)$ is an irreducible Riemannian symmetric space, $G$ is the
identity component of $\mathfrak I(M,Q)$.
\hfill \end{proof}

\begin{cor} \label{eq-cliff}
Let $(M,F)$ be a connected, simply connected irreducible GASBS of compact
type, and
$(M,Q)$ a Riemannian manifold whose Levi-Civit\` a connection coincides
with the linear connection of $(M,F)$, as in {\rm Remark \ref{szabo}}.
Let $\Gamma$ be a subgroup of $\mathfrak I(M,F)$.  Then
the centralizer of $\Gamma$ in $\mathfrak I(M,F)$ is
transitive on $M$ if and only if  the centralizer of $\Gamma$ in
$\mathfrak I(M,Q)$ is transitive on $M$.
\end{cor}
\begin{proof} Since $M$ is connected, the centralizer of $\Gamma$ in
$\mathfrak I(M,F)$ is transitive on $M$ if and only if its identity
component is transitive on $M$, and that happens if and only if the
centralizer of $\Gamma$ in the identity component $\mathfrak I(M,F)^0$ is
transitive on $M$.  Similarly the centralizer of $\Gamma$ in
$\mathfrak I(M,Q)$ is transitive on $M$ if and only if the
centralizer of $\Gamma$ in  $\mathfrak I(M,Q)^0$ is
transitive on $M$.  Since $\mathfrak I(M,F)^0 = \mathfrak I(M,Q)^0$
by Theorem \ref{equal-isom}, the
assertion follows.
\hfill \end{proof}

\begin{lem} \label{clif-clif}
Let $(M,F)$ be a connected, simply connected iirreducible GASBS of compact
type, and
$(M,Q)$ a Riemannian manifold whose Levi-Civit\` a connection coincides
with the linear connection of $(M,F)$, as in {\rm Remark \ref{szabo}}.
If $g$ is a Clifford translation of $(M,F)$ then $g$ is a
Clifford translation of $(M,Q)$.
\end{lem}
\begin{proof}  We follow the line of argument of \cite{OZ74}.
As we noted in Theorem \ref{oz}, if $x \in M$ then $g$ preserves
every minimizing $(M,F)$--geodesic from $x$ to $g(x)$.
Every such $(M,F)$--geodesic $t \mapsto \gamma_x(t)$ also is a
$(M,Q)$--geodesic.  Express $M$ as the coset space $G/K_x$ where
$G = \mathfrak I(M,F) \subset \mathfrak I(M,Q)$ and $K_x$ is its isotropy
subgroup at $x$.  Let $(\mathfrak g, \sigma_x)$ denote the Minkowski
symmetric Lie algebra for $(M,F)$ centered at $x$.  So
$\mathfrak g = \mathfrak k_x + \mathfrak m_x$
under the symmetry $\sigma_x$ at $x$, and the Lie algebra $\mathfrak k_x$
of $K_x$ is the fixed point set of $\sigma_x$ on $\mathfrak g$.
Since $(M,Q)$ is a Riemannian symmetric space we have
$g = \exp(\xi_x)k_x$ where $\xi_x \in \mathfrak m_x$ and $k_x \in K_x$, and
we can assume the parameterization of $\gamma_x$ to be such that
$\gamma_x(t) = \exp(t\xi_x)x$ for all $t$.  As $g$ preserves $\gamma_x$
in the sense of Definition \ref{def-pres}, Ad$(k_x)\xi_x = \xi_x$.  We
have proved the analog of \cite[Proposition 2.1(i)]{OZ74} for $(M,F)$.
There note Pres$(g)$ = $M$, and that the analog of \cite[Proposition 2.1(ii)]{OZ74}
is vacuous in our special situation.
\smallskip

The considerations of \cite{OZ74} following \cite[Proposition 2.1]{OZ74}
are purely group--theoretic, and somewhat simplified because here
Pres$(g)$ = $M$,  so we have the $(M,F)$--analogs of the remainder of
$\cite{OZ74}$ for our Clifford translation $g$.  In particular the
centralizer of $g$ in $\mathfrak I(M,F)$ is transitive on $M$.  Now the
centralizer of $g$ in $\mathfrak I(M,Q)$ is transitive on $M$, so $g$ is
a Clifford translation of $(M,Q)$.
\hfill \end{proof}
\medskip

\noindent
{\bf Proof of Theorem \ref{maintheorem}.}
By Theorem \ref{thm3.2}, the proof of Theorem \ref{maintheorem} is reduced
to the case where $(M,F)$ is irreducible and of compact type, i.e. where $\mathfrak I(M,F)$
is a compact semisimple Lie group.  There, Lemma \ref{clif-clif} says that
if the properly discontinuous group $\Gamma$ consists of Clifford translations
of $(M,F)$ then it consists of Clifford translations of $(M,Q)$.  By
\cite[Theorem 6.1]{WO62} its
centralizer in $\mathfrak I(M,Q)$ is transitive on $M$, and Corollary
\ref{eq-cliff} says that its centralizer in $\mathfrak I(M,F)$ is transitive
on $M$, so $(M,F)/\Gamma$ is homogeneous.  The remaining statement follows
directly from \cite[Theorem 6.2]{WO62}.


{\small
\begin{thebibliography}{99}

\bibitem{BCS00} D. Bao, S. S. Chern, \& Z. Shen, An Introduction to
Riemann-Finsler Geometry, Springer-Verlag, New York, 2000.

\bibitem{BN081} V. N. Berestovskii \& Yu. G. Nikonorov, Killing vector fields of constant length
on locally symmetric Riemannian manifolds, Transformation Groups, 13  (2008),
25-45.

\bibitem{BN082} V. N. Berestovskii \& Yu. G. Nikonorov, On $\delta$-homogeneous Riemannian manifolds, Diff. Geom. Appl., 26 (2008), 514-535.

\bibitem{BN09}  V. N. Berestovskii \& Yu. G. Nikonorov, Clifford-Wolf homogeneous Riemannian manifolds, Jour. Differ. Geom., 82 (2009), 467-500.

\bibitem{CH43} S. S. Chern, On the euclidean connections in a Finsler space, Proc. Nat. Acad. Sci. USA, 29 (1943), 33-37.

\bibitem{CR61} R. Crittenden, Minimal and conjugate points in symmetric spaces, Canad. J. Math., 14 (1962), 320-328.

 \bibitem{DP} S. Deng, Clifford-Wolf translations of Finsler spaces of  negative flag curvature,
preprint, [Math.DG] arXiv: 1204.5048.

\bibitem{DH04} S. Deng \& Z. Hou, Invariant Finsler metrics on homogeneous
 manifolds,  J.  Phys. A: Math. Gen., 37
(2004), 8245-8253.

\bibitem{DH05} S. Deng \& Z. Hou, Minkowski symmetric Lie algebras and symmetric Berwald spaces, Geometriae Dedicata, 113 (2005), 95-105.

\bibitem{DH07} S. Deng \& Z. Hou, On symmetric Finsler spaces, Israel J. Math., 166 (2007), 197-219.

\bibitem{DM1} S. Deng \& M. Xu, Clifford-Wolf translations of Finsler spaces, to appear in Forum Math., [Math.DG] arXiv:1201.3714v1.

\bibitem{DM2} S. Deng \& M. Xu, Clifford-Wolf translations of homogeneous Randers spheres, preprint, [Math. DG] arXiv: 1204.5232.

\bibitem{DM3} S. Deng \& M. Xu, Clifford-Wolf translations of left invariant Randers metrics on compact Lie groups, [Math. DG] arXiv: 1204.5233.

\bibitem{HE78} S. Helgason, Differerntial Geometry, Lie Groups and Symmetric Spaces, Academic Press, New York, 1978.

\bibitem{KN63} S. Kobayashi \& K. Nomizu, Foundations of Differential Geometry,
Volume I, Interscience Publishers, 1963.

\bibitem{OZ74} V. Ozols, Clifford translations of symmetric spaces, Proc. Amer. Math. Soc.,  44 (1974), 169-175.

\bibitem{SZ81} Z. I. Szab\'{o}, Positive definite Berwald spaces,
 Tensor  N. S.,  38 (1981), 25-39.

\bibitem{WO60} J. A. Wolf, Sur la classification des vari\' et\' es
riemanni\` ennes homog\` enes \`a courbure constante, C. R. Acad.
Sci. Paris,  250 (1960), 3443-3445.

\bibitem{WO62} J. A. Wolf, Locally symmetric homogeneous spaces, Commentarii Mathematici Helvetici,  37 (1962/63), 65-101.

\bibitem{WO64} J. A. Wolf, Homogeneity and bounded isometries in manifolds of negative curvature, Illinois J. Math., 8 (1964), 14-18.

\end{thebibliography}
}
\end{document}